\documentclass{tac}
\usepackage{xy}
\usepackage{amssymb}
\usepackage{plain}
\usepackage{amsfonts}
\usepackage[centertags]{amsmath}
\xyoption{all }

\def\mc#1{\mathcal {#1}}
\def\A{\mc A}
\def\B{\mc B}
\def\C{\mc C}
\def\D{\mc D}
\def\E{\mc E}
\def\F{\mc F}
\def\M{\mc M}
\def\N{\mc N}
\def\R{\mc R}
\def\S{\mc S}

 \author{S.N. Hosseini, A.R. Shir Ali Nasab, W. Tholen, L. Yeganeh}
\address{Mathematics Department\\ Shahid Bahonar University of Kerman\\ Kerman, Iran
\\ 
\\Department of Mathematics and Statistics\\ York University, Toronto, Canada}

 \title{Quotients of Span Categories That Are Allegories\\and the representation of regular categories}

 \keywords{span category, relation, stable system, regular category, (unitary, tabular) allegory, representation of allegories.\\
 \\
 { Acknowledgements:} The second author's research was in part supported by a grant from IPM. The third author acknowledges support under the Discovery Grants Program by the National Sciences and Engineering Council of Canada}
\amsclass{18B10, 18A32, 18E08}

 \eaddress{nhoseini@uk.ac.ir, ashirali@math.uk.ac.ir, tholen@yorku.ca, yeganehleila39@math.uk.ac.ir}

\begin{document}

 \maketitle

 \begin{abstract}
We consider the ordinary category $\mathsf{Span}(\mathcal C)$ of (isomorphism classes of) spans of morphisms in a category $\mathcal C$ with finite limits as needed, composed horizontally via pullback, and give a general criterion for a quotient of $\mathsf{Span}(\mathcal C)$ to be an allegory. In particular, when $\mathcal C$ carries a pullback-stable, but not necessarily proper,  $(\mathcal E,\mathcal M)$-factorization system, we establish a quotient category $\mathsf{Span}_{\mathcal E}(\mathcal C)$ that is isomorphic to the category $\mathsf{Rel}_{\mathcal M}(\mathcal C)$ of $\mathcal M$-relations in $\mathcal C$, and show that it is a (unitary and tabular) allegory precisely when $\mathcal M$ is a class of monomorphisms in $\mathcal C$. Without this restriction, one can still find a least pullback-stable and composition-closed class $\mathcal E_{\bullet}$ containing $\mathcal E$
such that $\mathsf{Span}_{\mathcal E_{\bullet}}(\mathcal C)$ is a unitary and tabular allegory. In this way one obtains a left adjoint to the 2-functor that assigns to every unitary and tabular allegory the regular category of its Lawverian maps. With the Freyd-Scedrov Representation Theorem for regular categories, we conclude that every finitely complete category with a stable factorization system has a reflection into the huge 2-category of all regular categories.
 \end{abstract}

\section{Introduction} 


$$ $$ 
 By identifying vertically isomorphic morphisms (= 1-cells) in Benabou's bicategory $\mathcal{S}\emph{pan}(\mathcal C)$ of spans of morphisms in a category $\mathcal C$ with pullbacks (see \cite{Benabou}, \cite{Borceux1}) one obtains the (ordinary) category
 $$\mathsf{Span}(\mathcal C),$$
 in which spans get composed horizontally via pullback in $\mathcal C$.
 If $\mathcal C$ is regular, so that $\mathcal C$ has also binary products and a stable (regular epi, mono)-factorization system, one may similarly form 
 the category $$\mathsf{Rel}(\mathcal C)$$ of sets and relations (= isomorphism classes of monic spans) in $\mathcal C$, with the horizontal composite of a composable pair of relations obtained as a regular image of their span composite. This category inherits from the bicategory $\mathcal{S}\emph{pan}(\mathcal C)$ the structure of a (strict) 2-category, with 2-cells given by order; actually, $\mathsf{Rel}(\mathcal C)$ is the prototypical example of a {\em (unitary and tabular) allegory}: see \cite{fs}, \cite{js}.
 
 More generally, as done in \cite{Pav 1} (with predecessors of this work presented under more restrictive conditions in \cite{Klein} and \cite{Meisen}, and with a weakening of the pullback-stability constraint given in \cite{JW 1996}),  without any epi- or mono restrictions one may consider an arbitrary stable factorization system $(\mathcal E,\mathcal M)$ of a category $\mathcal C$ with binary products and pullbacks and form the category $${\mathsf{Rel}}_{\mathcal M}(\mathcal C).$$ Its morphisms are represented by those spans $(A\longleftarrow R\longrightarrow B)$ whose induced morphism $R\longrightarrow A\times B$ lies in $\mathcal M$. In this paper we take a fresh look at this category, by treating it as a quotient category of 
  $\mathsf{Span}(\mathcal C)$. In fact, for any pullback-stable class $\mathcal E$ of morphisms in $\mathcal C$ that contains all isomorphisms and is closed under composition, we describe a compatible equivalence relation $\sim_{\mathcal E}$ on $\mathsf{Span}(\mathcal C)$ which renders its quotient category $$\mathsf{Span}_{\mathcal E}(\mathcal C)$$ isomorphic to $\mathsf{Rel}_{\mathcal M}(\mathcal C)$ when $\mathcal M$ is a factorization partner of $\mathcal E$; see Theorem \ref{RelasquotientofSpan}.
  
 Folklore knowledge says that ${\mathsf{Rel}}_{\mathcal M}(\mathcal C)$ is still an allegory, provided that $\M$ is a class of monomorphisms in 
$\C$; indeed, in Section 4 of this paper we give a novel proof of this fact. In \cite{Milius} it was shown that the provision $\mathcal M\subseteq \ 
\mathrm{Mono}(\mathcal C)$ is essential for obtaining an allegory; for example; the stable factorization system (bijective on objects, fully faithful) of the ordinary category 
${\sf Cat}$ of small categories, ${\sf Rel}_{\{\mathrm{fully\, faithful}\}}({\sf Cat})$ fails to be an allegory. 
In Theorem \ref{nec suf con rel be all} we will actually show that the provision $\M\subseteq {\mathrm{Mono}}(\C)$ is also necessary for ${\sf Rel}_{\M}(\C)$ to form an allegory.

This last result is a consequence of our answer to the more general question suggested by the title of the paper: When does a compatible  
equivalence relation $\sim$ on $\mathsf{Span}(\mathcal C)$ make its quotient category an allegory? The necessary and sufficient condition given in Theorem \ref{Spanallegory} may in particular be exploited in the case that $\sim$ is the relation $\sim_{\mathcal E}$ induced by a stable morphism class $\mathcal E$ as above (Theorem \ref{sim_F-allegory}), which then leads to the aforementioned Theorem \ref{nec suf con rel be all} whenever $\mathcal E$ belongs to a stable factorization system $(\mathcal E,\mathcal M)$ of $\mathcal C$.

Somewhat surprisingly, in the presence of a stable factorization system $(\mathcal E,\mathcal M)$ with $\M\subseteq {\mathrm{Mono}}(\C)$, the allegory
$\mathsf{Span}_{\mathcal E}(\mathcal C)$ is (in the terminology of \cite{fs}) already {\em unitary} and {\em tabular} (Corollary \ref{nec suf con rel be tball}).
This important addendum follows with the help of the principal result of the paper:  Given any stable factorization system $(\mathcal E,\mathcal M)$ in a finitely complete category $\mathcal C$, there is a least stable and composition-closed class $\mathcal E_{\bullet}$ containing
 $\mathcal E$ that makes $\mathsf{Span}_{{\mathcal E}_{\bullet}}(\mathcal C)$ a unitary tabular allegory (Theorem \ref{tabulation}). 
 
 The global significance of this result is outlined in Section 6, where we set up the huge 2-categories of unitary tabular allegories on one hand, and of finitely complete categories equipped with a stable factorization system on the other hand. We then show that the construction of Theorem \ref{tabulation} gives us a left adjoint to the 2-functor
 $$\mathsf{Map}:\mathfrak{UTabAll}\longrightarrow \mathfrak{StabFact},$$ 
 which assigns to a unitary tabular allegory its category of (Lawverian) maps, equipped with its stable factorization system that makes it a regular category. With the {\em Freyd-Scedrov Representation Theorem} \cite[2.154]{fs}, which lets us present $\mathfrak{UTabAll}$ as 2-equivalent to the full subcategory $\mathfrak{RegCat}$ of $\mathfrak{StabFact}$, given by all regular categories, we see that every finitely complete category with a stable factorization system allows for a reflection into $\mathfrak{RegCat}$.

\section{Categories of relations as quotients of the span category}
{\em Throughout this paper, we let $\C$ be a category with binary products and pullbacks}.   

For objects $A,B$ in $\C$, a {\em{span}} $(f,g)$ in $\C$ with domain $A$ and codomain $B$ is given by a pair of morphisms
\begin{center}
	$\xymatrix{A & D\ar[l]_{f}\ar[r]^{g} & B}\;.$
\end{center}
These are the objects of the category $\mathcal{S}\emph{pan}(\mathcal C)(A,B)$
whose morphisms $u:(f, g)\longrightarrow(\tilde{f},\tilde{g})$ are given by $\C$-morphisms $u$ with $\tilde{f} u=f$ and 
$\tilde{g} u=g$, to be composed vertically as in $\C$.
\begin{center}
	$\xymatrix{& & D\ar[lld]_{f}\ar[dd]_{u}\ar[rrd]^{g} & & \\
		A & & & & B\\
		& & \tilde{D}\ar[llu]^{\tilde{f}}\ar[rru]_{\tilde{g}} & &}$
\end{center}
$\mathcal{S}\emph{pan}(\mathcal C)$ becomes a bicategory when one composes the spans $(f,g):A\longrightarrow B$ and  $(h,k):B\longrightarrow C$ horizontally via (a tacitly chosen) pullback in the usual fashion, as shown by
\begin{center}
	$\xymatrix{&& P\ar[ld]_{h'}\ar[rd]^{g'} && \\
		& D\ar[ld]_{f}\ar[rd]^{g}\ar@{}[rr]|{\rm{pb}} & & E\ar[ld]_{h}\ar[rd]^{k} & \\
		A && B && C}$
\end{center}
that is: $$(h,k)\circ(f,g) := (f h',k g').$$ 

In this paper we are mostly interested in the quotient categories of the ordinary category
 $$\sf{Span}(\C)$$
 whose objects are those of $\C$, and whose morphisms are (vertical) isomorphism classes of spans in $\C$, with their horizontal composition. For simplicity, we denote the (vertical) isomorphism class of $(f,g)$ again by $(f,g)$.
 The category $\sf{Span}(\C)$  comes with the involution $$(f,g)^{\circ}=(g,f),$$ 
 making it
 a self-dual category.

 
 
 
  For a class $\E$ of morphisms in $\C$ we define:
 
\begin{definition}\label{Fequivalence}
{\em (1)} The class $\E$ is a {\em stable system} in $\C$ if it 
contains the isomorphisms of $\C$ and is closed under composition and stable under pullback in $\C$.

{\em (2)} For the stable system $\E$ in $\C$ and spans $(f,g),(\tilde{f},\tilde{g}):A\to B$, using the vertical structure of the bicategory $\mathcal{S}\text{pan}(\mathcal C)$, we define the relation $\leq_{\E}$ by
$$(f,g)\leq_{\E}(\tilde{f},\tilde{g})\iff\exists\, x\in\E\;(x:(f,g)\to(\tilde{f},\tilde{g}) \;\emph{in}\; \mathcal{S}pan(\mathcal C)).$$

{\em (3)} 
Since the stable system $\E$ contains the isomorphisms of $\C$ and is closed under composition, $\leq_{\E}$ is reflexive and transitive. The least equivalence relation on ${\sf Span}(\C)$ containing $\leq_{\E}$ is denoted by $\sim_{\E}$; it is described by $\leq_{\E}$-zigzags:	
$$(f, g)\sim_{\E}(\tilde{f},\tilde{g})\iff \exists\; s\geq1, (f_0, g_0),...,(f_s, g_s):$$ $$(f, g)=(f_0, g_0)\leq_{\E}(f_1, g_1)\geq_{\E}(f_2, g_2)\leq_{\E} ... \geq_{\E}(f_s, g_s)=(\tilde{f},\tilde{g}).$$
\end{definition}
We denote the $\sim_{\E}$-equivalence classes of $(f, g)$ by $[f, g]_{\E}$. (Note that, according to our notational convention above, for $[f,g]_{\mathrm{Iso}(\mathcal C)}$ we write just $(f,g)$.)  The horizontal composition of two classes $\xymatrix{A\ar[r]^{[f, g]_{\E}}&B\ar[r]^{[h, k]_{\E}}&C}$ is defined by the horizontal composition of their representatives:
$$[h, k]_{\E}\circ[f, g]_{\E}:=[(h, k)\circ (f, g)]_{\E}=[fh', kg']_{\E}.$$ To make sure that the quotient category
$${\sf Span}_{\E}(\C):={\sf Span}(\C)/\sim_{\E}  $$ is well defined, 
we have to show:

\begin{proposition}\label{basic lemma}
	For a stable system $\E$ in $\C$, the equivalence relation $\sim_{\E}$ is compatible with the (horizontal) composition of morphisms in ${\sf Span}(\C)$.
\end{proposition}
\begin{proof}
	Since we have the isomorphism $(-)^{\circ}$, we only need to show that, for $(f,g),(m,n):A\to B$ and $(h,k):B\to C$, one has the implication
	$$(f,g)\sim_{\E}(m,n)\Longrightarrow (h,k)\circ (f,g)\sim_{\E}(h,k)\circ (m,n).$$
	In fact, it is sufficient to show this implication when $\sim_{\E}$ is replaced by $\leq_{\E}$. Hence, we just show:
	$$(f,g)\leq_{\E}(\tilde{f},\tilde{g}):A\to B,\;(h,k):B\to C\Longrightarrow(h,k)\circ(f,g)\leq_{\E}(h,k)\circ(\tilde{f},\tilde{g}).$$
	To this end, with $x:(f,g)\to(\tilde{f},\tilde{g})$ for $x\in \E$, and with $(h',g')={\rm pb}(g,h)$ as above and $(\tilde{h}',\tilde{g}')={\rm pb}(\tilde{g},h)$, one obtains a unique morphism $x'$ making the diagram
	\begin{center}
		$\xymatrix{P\ar[r]_{x'}\ar[d]_{h'}\ar@/^1.0pc/[rr]^{g'} & \tilde{P}\ar[r]_{\tilde{g}'}\ar[d]^{\tilde{h}'} & E\ar[d]^{h}\\
			D\ar[r]^{x}\ar@/_1.0pc/[rr]_{g} & \tilde{D}\ar[r]^{\tilde{g}} & B\\
		}$	
	\end{center}
	commute. With the outer rectangle and the right square being pullback diagrams, so is the left square. Hence, as a pullback	 of $x$, also $x'$ lies in $\E$, which shows $(h,k)\circ(f,g)\leq_{\E}(h,k)\circ(\tilde{f},\tilde{g})$.
\end{proof}
If the morphism class $\E$ belongs to a (pullback-)stable factorization system $(\E,\M)$ in $C$, then $\E$ is certainly a stable system, and we may construct the quotient category  ${\sf Span}_{\E}(\C)$. On the other hand, granted the existence of binary products in $\C$, still {\em without} assuming that $\M$ be a class of monomorphisms in $C$, we may also form the category  
$$\sf{Rel}_{\M}(\C)$$
of {\em $\M$-relations} in $\C$ (see \cite{Milius}); its objects are the objects of $\C$, and a morphism $(f,g):A\rightarrow B$ is (the isomorphism class of) a span in $\C$ such that $\langle f,g\rangle :D\rightarrow A\times B$ lies in $\mathcal M$; its composite with 
$(h,k):B\rightarrow C$ in $\sf{Rel}_{\M}(\C)$ is defined by $$(h,k)\cdot(f,g):=(p_1 m,p_2 m),$$ where $m$ is the $\mathcal M$-part of the $(\E,\M)$-factorization of $\langle fh',kg'\rangle :P\rightarrow A\times C$, with $p_1,p_2$ denoting the projections of the product $A\times C$.

\begin{center}
	$\xymatrix{&&&& P\ar[lld]_{h'}\ar[rrd]^{g'} &&&&\\
		&& D\ar@{}[rrrr]|{pb}\ar[lld]_{f}\ar[rrd]^{g} &&&& E\ar[lld]_{h}\ar[rrd]^{k} &&\\
		A && && B &&&& C }$
\end{center}
We note that, for a span $(f,g)$ with the decomposition $\langle f,g\rangle=me$ where $m\in \M$ and $e\in \E$, one clearly has $[f,g]_{\E}=[\pi_1 m,\pi_2 m]_{\E}$, where $\pi_1$ and $\pi_2$ are the projections of $A\times B$.
\begin{theorem}\label{RelasquotientofSpan}
For a stable factorization system $(\E, \M)$ in a finitely complete category $\C$, the categories	$$\sf{Rel}_{\M}(\C)\cong\sf{Span}_{\E}(\C)$$ are isomorphic. If $\C$ is a regular category, with $\E=\mathrm{RegEpi}(\C)$, $\M=\mathrm{Mono}(\C)$ and $\sf{Rel}(\C)=\sf{Rel}_\M(\C)$, one has in particular $$\sf{Rel}(\C)\cong\sf{Span}_\E(\C).$$
\end{theorem}
\begin{proof}
For every $\M$-relation $(f,g):A\to B$ one has the induced $\sim_{\E}$-equivalence class $S(f,g)=[f,g]_{\E}$ of the span $(f,g)$. With the note above the Theorem it follows immediately that one obtains a functor
	$ S:\sf{Rel}_{\M}(\C)\longrightarrow\sf{Span}_{\E}(\C)$ which maps objects identically. 
	
	Conversely we show that taking $[f,g]_{\mathcal E}:A\rightarrow B$ to the $\M$-relation $R[f,g]_{\E}=(\pi_1m,\pi_2m):A\rightarrow B$,
	with $m$ the $\mathcal M$-part of $\langle f,g\rangle $, defines a functor
	$R:\sf{Span}_{\E}(\C)\longrightarrow\sf{Rel}_{\M}(\C)$. Indeed,
		to show that $R$ is well-defined, we suppose that $(f,g)\leq_{\mathcal E}(h,k)$, so that $he=f$ and $ke=g$ for some $e\in\E$. This means $\langle f,g\rangle =\langle h,k\rangle e$, which implies that the $\mathcal{M}$-part of $\langle f,g\rangle $ is isomorphic to the $\mathcal M$-part of $\langle h,k\rangle $.
		
	To show that $R$ preserves the composition, let $\xymatrix{A\ar[r]^{[f,g]_{\E}} & B\ar[r]^{[h,k]_{\E}} & C}$ be a pair of morphisms in $\sf{Span}_{\E}(\C)$ and $\langle f,g\rangle =me$ and $\langle h,k\rangle =m'e'$ for $e,e'\in \mathcal E$ and $m,m'\in \mathcal M$. In the following diagram, $f_1,g_1,h_1,k_1$ are defined to make the bottom triangles with the product projections $\pi_1,\pi_2,p_1,p_2$  commute, and $g',h',g_1',h_1'$ are such that $gh'=hg'$ and $g_1h_1'=h_1g_1'$ become pullback squares.
	\begin{center}
		$
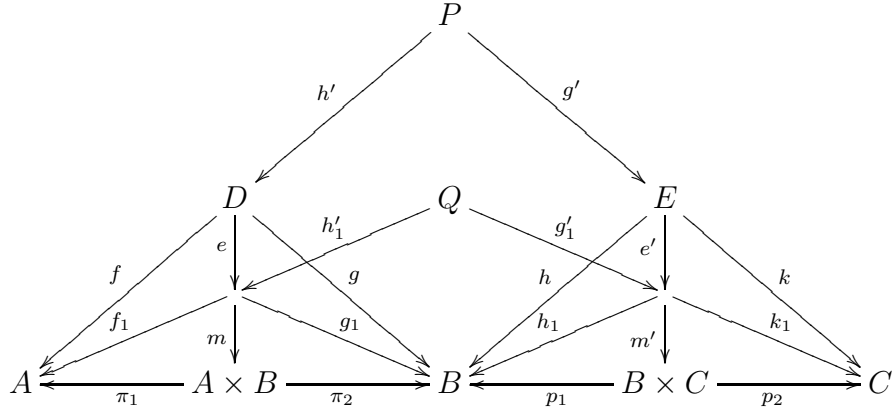
\xymatrix{&&&& P \ar[lldd]_{h'}\ar[rrdd]^{g'} &&&&\\
			&&&&&&&&\\
			&& D\ar[lldd]_{f}\ar[d]_{e}\ar[rrdd]^{g} &&Q \ar[lld]_{h_1'}\ar[rrd]^{g_1'}&& E\ar[lldd]_{h}\ar[d]_{e'}\ar[rrdd]^{k} &&\\
			&& \ar[d]_{m}\ar[lld]_{f_1}\ar[rrd]^{g_1} &&&& \ar[d]_{m'}\ar[lld]_{h_1}\ar[rrd]^{k_1} &&\\
			A && A\times B\ar[ll]^{\pi_1}\ar[rr]_{\pi_2} && B && B\times C\ar[ll]^{p_1}\ar[rr]_{p_2} && C}$
	\end{center}
	Let $\xymatrix{A & A\times C\ar[l]_{q_1}\ar[r]^{q_2} & C}$ be product projections and $n$ be the $\mathcal M$-part of the morphism $\langle f_1h_1',k_1g_1'\rangle :Q\rightarrow A\times C$. Then we have 
	$$[h,k]_{\mathcal E}\circ[f,g]_{\mathcal E}=[h_1,k_1]_{\mathcal E}\circ[f_1,g_1]_{\mathcal E}=[f_1h'_1,k_1g'_1]_{\E}=[q_1 n,q_2 n]_{\mathcal E},$$
	and the functoriality of $R$ may be seen easily now:
	\begin{equation*}
	R([h,k]_{\mathcal E}\circ[f,g]_{\mathcal E}) = R[q_1 n,q_2 n]_{\mathcal E}=(q_1 n,q_2 n)=(h_1,k_1)\cdot(f_1,g_1) 
	=R[h,k]_{\mathcal E}\cdot R[f,g]_{\mathcal E}.
	\end{equation*}
The identities $ RS=\mathrm{Id}$ and $SR=\mathrm{Id}$ follow directly from the definitions of $S$ and $R$.

\end{proof}

\section{Allegories arising as quotients of span categories}
Recall that an {\em allegory} \cite{fs} is a category $\A$ such that 
\begin{itemize}
\item $\A$ is poset-enriched, so that each hom-set carries a partial order which is preserved by the composition $\cdot$ of $\A$; 
\item $\A$ carries an involution $(-)^{\circ}:\A^{\rm op}\to \A$, {\em{i.e.}}, a contravariant self-inverse functor; 
\item binary infima exist in each hom-set, and the involution preserves them;
\item {\em Freyd's modular law} holds in $\A$:
$$ (s\cdot r)\wedge t\leq s\cdot (r\wedge (s^{\circ}\cdot t)),$$
for all $r:A\to B,\, s:B\to C,\, t:A\to C$ in $\A$.
\end{itemize}
Of the various derived rules valid in an allegory, we will use in particular the {\em special modular law}:
$$r\leq r\cdot r^{\circ}\cdot r$$
for all morphisms $r$ in $\A$ (see \cite[2.112]{fs}); indeed, with Freyd's modular law, one has $$r\leq (r\cdot 1)\wedge r\leq r\cdot(1\wedge(r^{\circ}\cdot r))\leq r\cdot(r^{\circ}\cdot r).$$

As mentioned in the Introduction, the standard example of an allegory is the category ${\sf Rel}(\C)$ for a  
regular category $\C$, and this fact may be generalized to the case of a stable factorization system $(\E,\M)$ in $\C$, provided that $\M$ is a class of monomorphisms in 
$\C$; indeed, Corollary \ref{F contains Ret} provides a novel proof that ${\sf Rel}_{\M}(\C)$ is an allegory under this provision. 
In Theorem \ref{nec suf con rel be all} we prove that the provision $\M\subseteq {\rm{Mono}}(\C)$ is actually necessary for ${\sf Rel}_{\M}(\C)$ being an allegory.

Since, in the presence of a stable factorization system, ${\sf Rel}_{\M}(\C)$ has been presented as a quotient of ${\sf Span}(\C)$ (see Theorem \ref{RelasquotientofSpan}), without any reference to such a system it is natural to ask {\em what it takes for an equivalence relation 
$\sim$ on ${\sf Span}(\C)$ to make the quotient category with its horizontal composition an allegory}. 

To answer this question, we first recall that one trivially has the involution
$${\sf Span}(\C)^{\rm op}\to{\sf Span}(\C),\;(f,g)\mapsto(f,g)^{\circ}=(g,f).$$
Next, the hom-sets of ${\sf Span}(\C)$ inherit from the local categories of the bicategory $\mathcal{S}\emph{pan}(\mathcal C)$ a (up to isomorphism) commutative and associative operation which, by the isomorphism 
$$\mathcal{S}\emph{pan}(\mathcal C)(A,B) \cong \C/A\times B$$
for all objects $A,B$ in $\C$, is given by the direct product in the comma category on the right.
Denoting the direct product in $\C/A\times B$
of (the induced arrows to $A\times B$ of) spans $(f,g),(p,q):A\to B$ by $(f,g)\wedge(p,q)$, we observe that it is constructed by pullback in $\C$, as 
$$(f,g)\wedge(p,q)=(\pi_1d,\pi_2d),\quad{\rm with}\quad d=\langle f,g\rangle \rho_1=\langle p,q\rangle\rho_2,$$
where $\rho_1,\rho_2$ are pullback projections and $\pi_1,\pi_2$ are product projections in $\C$.
\begin{equation}\label{def of meet in span}
\xymatrix{&& P\ar[ld]_{\rho_1}\ar[rd]^{\rho_2}\ar[dd]_d && \\
	& D\ar[ld]_{f}\ar[rd]^{\langle f, g\rangle} & & F\ar[ld]_{\langle p,q\rangle}\ar[rd]^{q} & \\
	A && A\times B\ar[ll]_{\pi_1}\ar[rr]^{\pi_2} && B}
\end{equation}
As an operation on the hom-sets of ${\sf Span}(\C)$, $\wedge$ is (strictly) commutative and associative; furthermore, $(\pi_1,\pi_2)$ is neutral in ${\sf Span}(\C)(A,B)$ with respect to $\wedge$. More importantly,
the diagram above shows immediately that $(f,g)\wedge(p,q)$ may be presented as a horizontal composite, as in formula (1) of the following proposition, which also lists three further important observations.
\begin{proposition}\label{preallegory}
For all spans $(f,g),(p,q):A\to B,\; (h,k):B\to C,\;(m,n):A\to C$ one has:
\begin{itemize}
\item[\em(1)] $(f,g)\wedge(p,q)=(\langle p,q\rangle,q)\circ(f,\langle f,g\rangle)$ {\em (meet via composition)};
\item[\em(2)] $(f,f)\wedge(1_A,1_A)=(f,f)$ {\em (meet neutrality of identity morphisms)};
\item[\em(3)] there is a canonical 2-cell $$d:(h,k)\circ((f,g)\wedge(p,q))\longrightarrow((h,k)\circ(f,g))\wedge((h,k)\circ(p,q))$$
in $\S\text{pan}(\C)$, witnessing the {\em lax distributivity of the composition over meet};
\item[\em(4)] there is a canonical 2-cell $$c: ((h,k)\circ(f,g))\wedge(m,n)\longrightarrow(h,k)\circ((f,g)\wedge((h,k)^{\circ}\circ(m,n)))$$
in $\S\text{pan}(\C)$, which we call the {\em Freyd modularity morphism} of the given spans.
\end{itemize}
\end{proposition}
\begin{proof}
	Formula (1) follows from the above diagram, and (2) follows easily from the definition. For proving (3), we let the composites $(h,k)\circ(
f,g)$ and $(h,k)\circ(p,q)$ be constructed by the pullback squares of the diagrams
\begin{center}
$\xymatrix{&& P\ar[ld]_{h'}\ar[rd]^{g'} &&\\
&D\ar[rd]_{g}\ar[ld]_{f} && E\ar[ld]^{h}\ar[rd]^{k}&\\
A && B && C\\
}$	
\hfil$\xymatrix{&& Q\ar[rd]^{q'}\ar[ld]_{h''} &&\\
&F\ar[rd]_{q}\ar[ld]_{p} && E\ar[ld]^{h}\ar[rd]^{k} &\\
A && B && C\\
}$	
\end{center}
and the meets $(f,g)\wedge(p,q)$ and $((h,k)\circ(f,g))\wedge((h,k)\circ(p,q)))$ by the pullback diagrams
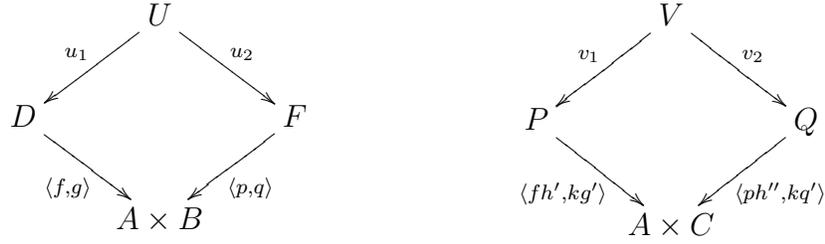
\begin{center}
$\xymatrix{& U\ar[ld]_{u_1}\ar[rd]^{u_2} &\\
D\ar[rd]_{\langle f,g\rangle} && F\ar[ld]^{\langle p,q\rangle}\\
& A\times B &
}$	
\hfil$\xymatrix{& V\ar[rd]^{v_2}\ar[ld]_{v_1} &\\
P\ar[rd]_{\langle fh',kg'\rangle} && Q\ar[ld]^{\langle ph'',kq'\rangle}\\
& A\times C &\
}$	
\end{center}
With the composite $(h,k)\circ((f,g)\wedge(p,q))$ given by 
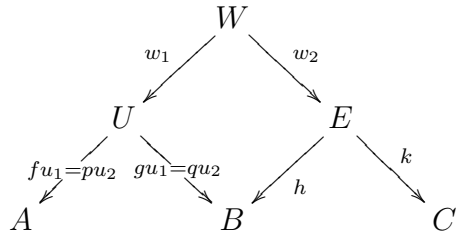
\begin{center}
$\xymatrix{&& W\ar[ld]_{w_1}\ar[rd]^{w_2} &&\\
&U\ar[rd]|{gu_1=qu_2}\ar[ld]|{fu_1=pu_2} && E\ar[ld]^{h}\ar[rd]^{k}&\\
A && B && C\\
}$	
\end{center}
one first obtains the morphisms $x:W\to P,\;y:W\to Q$, uniquely determined by $h'x=u_1w_1,\, g'x=w_2,\,h''y=u_2w_1,\,q'y=w_2$, and then the unique morphism $d:W\to V$ 
with $v_1d=x,\,v_2d=y$. One readily checks that $d$ is the desired 2-cell of spans.
 
For showing (4), we let the composites $(h,k)\circ(
f,g)$ and $(h,k)^{\circ}\circ(m,n)$ be constructed by the pullback squares of the diagrams
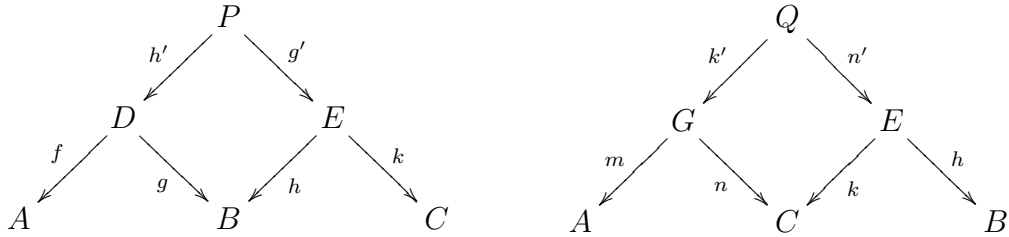
\begin{center}
$\xymatrix{&& P\ar[ld]_{h'}\ar[rd]^{g'} &&\\
&D\ar[rd]_{g}\ar[ld]_{f} && E\ar[ld]^{h}\ar[rd]^{k}&\\
A&& B &&C\\
}$	
\hfil$\xymatrix{&& Q\ar[rd]^{n'}\ar[ld]_{k'} &&\\
&G\ar[rd]_{n}\ar[ld]_{m} && E\ar[ld]^{k}\ar[rd]^{h}&\\
A&& C&&B\\
}$	
\end{center}
and the meets $((h,k)\circ(f,g))\wedge(m,n)$ and $(f,g)\wedge((h,k)^{\circ}\circ(m,n))$ by the pullback diagrams
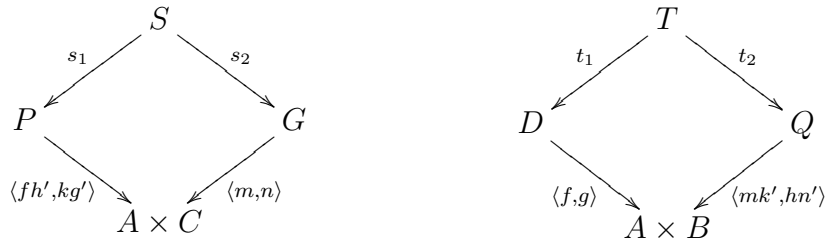
\begin{center}
$\xymatrix{& S\ar[ld]_{s_1}\ar[rd]^{s_2} &\\
P\ar[rd]_{\langle fh',kg'\rangle} && G\ar[ld]^{\langle m,n\rangle}\\
& A\times C &
}$	
\hfil$\xymatrix{& T\ar[rd]^{t_2}\ar[ld]_{t_1} &\\
D\ar[rd]_{\langle f,g\rangle} && Q\ar[ld]^{\langle mk',hn'\rangle}\\
& A\times B &\
}$	
\end{center}
One readily verifies that there is a unique morphism $a:S\to Q$ satisfying $k'a=s_2,\, n'a=g's_1$, and then a unique morphism $b:S\to T$ with $t_1b=h's_1,\,t_2b=a$. With the diagram
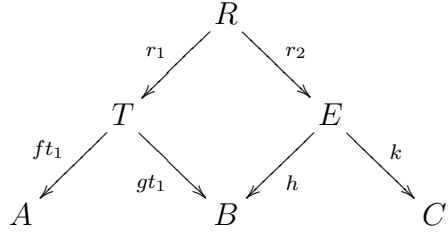
\begin{center}
$\xymatrix{&& R\ar[ld]_{r_1}\ar[rd]^{r_2} &&\\
&T\ar[rd]_{gt_1}\ar[ld]_{ft_1} && E\ar[ld]^{h}\ar[rd]^{k}&\\
A&& B &&C\\
}$	
\end{center}
displaying the span $(h,k)\circ((f,g)\wedge((h,k)^{\circ}\circ(m,n)))$, we finally obtain the unique morphism $c:S\to R$ in $\C$ with $r_1c=b,\,r_2c=n'a$. The identities $(ft_1r_1)c=ft_1b=fh's$ and $(kr_2)c=kn'a=kg's_1=ns_2$ show that $c$ is in fact the desired 2-cell of spans.
\end{proof}

We will return to items (2--4) of Proposition \ref{preallegory} only in Theorem \ref{Spanallegory}. Item (1) helps us to first specify the properties of an equivalence relation for spans which will enable us to form a quotient category that inherits the operations $(-)^{\circ}$ and $\wedge$, as follows.
\begin{definition} Let $\sim$ be a relation for spans in $\C$ such that only spans with the same domain and codomain may be related, and that vertically isomorphic spans are always related (so that $\sim$ is reflexive for the isomorphism classes of spans). We call $\sim$
\begin{itemize}
\item {\em inv(olution)-compatible} if $(f,g)\sim(\tilde{f},\tilde{g})$ implies $(g,f)\sim(\tilde{g},\tilde{f})$; 
 \item {\em comp(osition)-compatible} if $(f,g)\sim(\tilde{f},\tilde{g})$ implies $(h,k)\circ(f,g)\sim(h,k)\circ(\tilde{f},\tilde{g})$;
	\item {\em meet-compatible} if $(f,g)\sim(\tilde{f},\tilde{g})$ implies $(f,\langle f,g\rangle)\sim(\tilde{f},\langle\tilde{f},\tilde{g}\rangle)$,
\end{itemize}
whenever $(f,g),(\tilde{f},\tilde{g}):A\to B,\,(h,k):B\to C$. If $\sim$ enjoys all three properties, we will simply say that $\sim$ is {\em compatible}. 
\end{definition}

\noindent When $\sim$ is inv-compatible, it is easy to see that then  $\sim$ is
\begin{itemize}
	\item comp-compatible if, and only if, $(h,k)\sim(\tilde{h},\tilde{k})$ implies $(h,k)\circ(f,g)\sim(\tilde{h},\tilde{k})\circ(f,g), $
	 \end{itemize}
	 and when $\sim$ is inv- and comp-compatible, then $\sim$ is
	 \begin{itemize}
	\item meet-compatible if, and only if, $(f,g)\sim(\tilde{f},\tilde{g})$ implies $(\langle f,g\rangle,g)\sim(\langle \tilde{f},\tilde{g}\rangle,\tilde{g})$.	
	\end{itemize}
	For the last claim, exploit the defining property for $(g,f)\sim(\tilde{g},\tilde{f})$ and observe that one has $(g, \langle f,g\rangle)=(1_{B\times A}, t)\circ(g,\langle g,f\rangle)$, with the ``twist'' morphism $t:B\times A\to A\times B$.
		
It is a routine exercise, and a fact used frequently in this paper, that {\em each of the three compatibility properties is inherited from the given relation by the least equivalence relation generated by it}. 

 For a compatible equivalence relation $\sim$ for spans in $\C$, we denote the $\sim$-equivalence class of $(f,g)$ by $[f,g]_{\sim}$, or just by $[f,g]$ when the context makes it clear which relation $\sim$ we are referring to. We can now define
 \begin{itemize}
 \item $[f,g]^{\circ}:=[g,f]$;
 \item $[h,k]\circ[f,g]:=[(h,k)\circ(f,g)]$;
 \item $[f,g]\wedge[p,q]:=[(f,g)\wedge(p,q)]$.
 \end{itemize}	
One can easily verify that the above operations are well-defined for $\sim$-equivalence classes. Writing
$$\sf{Span}_{\sim}(\C)$$ 
for the resulting quotient category with its horizontal composition, we can now summarize our observations, as follows:
\begin{proposition}\label{almostallegory}
For a compatible equivalence relation $\sim$ for spans in $\C$, the quotient category  	$\sf{Span}_{\sim}(\C)$ carries an involution $(-)^{\circ}:\sf{Span}_{\sim}(\C)^{\rm op}\to\sf{Span}_{\sim}(\C)$ and a commutative, associative binary operation $\wedge$ that is preserved by the involution. The two operations are determined by the property that the projection functor ${\sf Span}(\C)\to{\sf Span}_{\sim}(\C)$ becomes a homomorphism with respect to them.
\end{proposition}

We can now narrow down the question posed at the beginning of this section and ask, under which condition on a compatible equivalence relation $\sim$ the quotient category ${\sf Span}_{\sim}(\C)$	 may become an allegory when equipped with the involution and the meet operation as described by the proposition. For the meet operation to induce a partial order on the hom-sets one needs the idempotency law $[f,g]\wedge[f,g]=[f,g]$ to hold in ${\sf Span}_{\sim}(\C)$. The following definition describes an easy sufficient condition for that law to hold, which will turn out to be necessary as well.

\begin{definition}\label{defallegoricalrelation}
We call a relation $\sim$ for $\C$-spans {\em allegorical} if 
$$(1_A,f)\sim(f,f)\circ(1_A,f)$$
for all morphisms $f:A\to B$ in $\C$. If $\sim$ is inv-compatible, this condition is equivalent to
$$(f,1_A)\sim(f,1_A)\circ(f,f)$$	
for all morphisms $f$. Note that, with the {\em graph functor}
$$\Gamma:\C\to{\sf Span}(\C),\; (f:A\to B)\longmapsto((1_A,f):A\to B),$$	
these conditions may also be written as 
$\Gamma f\sim\Gamma f\circ(\Gamma f)^{\circ}\circ\Gamma f$ or
$(\Gamma f)^{\circ}\sim(\Gamma f)^{\circ}\circ\Gamma f\circ(\Gamma f)^{\circ},$
simply because every span $(f,g)$ may be written as $$(f,g)=\Gamma g\circ(\Gamma f)^{\circ}.$$
\end{definition}
The following lemma is fundamental for our further investigations:
\begin{lemma}\label{meetidempotent}
Let the relation $\sim$ for $\C$-spans be inv- and comp-compatible and allegorical. If there exists a 2-cell $u:(f,g)\to(p,q)$ in $\S \text{pan}(\C)$, then
$$(f,g)\wedge(p,q)\sim(f,g).$$ 
In particular, for all spans $(f,g)$ one has $(f,g)\wedge(f,g)\sim(f,g)$.	
\end{lemma}
\begin{proof}
Putting $r=\langle f, g\rangle :D\to A\times B,\;s=\langle p,q\rangle:E\to A\times B$, so that $su=r$, and denoting by $\pi_2:A\times B\to B$ the product projection, under the assumptions on  $\sim$ we obtain
\begin{align*}	
 (f,g)\wedge(p,q) & =(s,q)\circ(f,r)\\	
   & = (1_{A\times B},\pi_2)\circ(s,s)\circ(1_E,s)\circ(f,u)\\
   & \sim (1_{A\times B},\pi_2)\circ(1_E,s)\circ(f,u)\\
   & = (1_E,q)\circ(f,u)\\
   & = (f,g).
\end{align*}
\end{proof}
With Proposition \ref{preallegory}, our Lemma \ref{meetidempotent} facilitates the proof of the following theorem.
\begin{theorem}\label{Spanallegory} For a compatible equivalence relation $\sim$ for spans, the operations $(-)^{\circ}$ and $\wedge$ of Proposition {\em \ref{almostallegory}} make ${\sf Span}_{\sim}(\C)$ an allegory if, and only if, $\sim$ is allegorical.	
\end{theorem}
\begin{proof}
Suppose first that ${\sf Span}_{\sim}(\C)$ is an allegory, with the operations as described. Hence, its hom-sets carry the partial order induced by $\wedge$, that is: 
$$(\ast)\qquad\qquad[f,g]\leq[m,n]\Longleftrightarrow[f,g]\wedge[m,n]=[f,g].$$	 
By first applying the special modular law and then Proposition \ref{preallegory}(2) we obtain for all $f:A\to B$ in $\C$ the inequalities
\begin{align*}
[1_A,f] & \leq [1_A,f]\circ[1_A,f]^{\circ}\circ[1_A,f]\\
& = [f,f]\circ[1_A,f]\\
& =	([f,f]\wedge[1_B,1_B])\circ[1_A,f]\\
& \leq [1_B,1_B]\circ[1_A,f]\\
& = [1_A,f].
\end{align*}
Consequently, $(1_A,f)\sim(f,f)\circ(1_A,f)$, as desired.

Conversely, assuming $\sim$ to be allegorical, by Proposition \ref{almostallegory} and Lemma \ref{meetidempotent}, 
${\sf Span}_{\sim}(\C)$ carries the involution 
$(-)^{\circ}$ and the operation 
$\wedge$ which induces a partial order on the hom-sets, as in $(\ast)$ above.
Furthermore, from Proposition \ref{preallegory}(3) and Lemma \ref{meetidempotent} one obtains
$$[h,k]\circ([f,g]\wedge[p,q])\leq([h,k]\circ[f,g])\wedge([h,k]\circ[p,q]),$$
which renders the composition of ${\sf Span}_{\sim}(\C)$ monotone. Likewise, Proposition \ref{preallegory}(4) and Lemma \ref{meetidempotent} show that Freyd's modular law holds in ${\sf Span}_{\sim}(\C)$.
\end{proof}
The existence of a 2-cell $u:(f,g)\to(p,q)$ as in Lemma \ref{meetidempotent} describes the natural preorder for spans, seen as the objects of the local categories of the bicategory $\S\text{pan}(\C)$. Considering their reflection onto posets produces the largest allegorical quotient of ${\sf Span}(\C)$, as follows:

\begin{corollary}\label{leastallegorical}
There is a least compatible and allegorical equivalence relation $\approx $ on ${\sf Span}(\C)$, described by
$$(f,g)\approx(p,q)\Longleftrightarrow\exists\, u, v\;(u:(f,g)\to(p,q),\;v:(p,q)\to(f,g)).$$
Consequently, for every compatible and allegorical relation $\sim$ on ${\sf Span}(\C)$,  one has a projection functor
$${\sf Span}_{\approx}(\C)\to{\sf Span}_{\sim}(\C),\,[f,g]_{\approx}\mapsto[f,g]_{\sim},$$
which maps objects identically and is a morphism of allegories, i.e., it preserves the involution and the binary meet.
\end{corollary}
\begin{proof}
One routinely shows that, when defining the relation $\approx$ as described, one obtains a compatible and allegorical equivalence relation. Considering any compatible and allegorical equivalence relation $\sim$, with Lemma \ref{meetidempotent} one obtains immediately the implication
$$(f,g)\approx(p,q)\Longrightarrow(f,g)\sim(p,q),$$
which confirms the first claim. The second claim is an obvious consequence of the first one.
\end{proof}	
\section{Allegories arising from stable systems}
As an easy exercise one shows that for every stable system $\E$ in $\C$ the equivalence relation $\sim_{\E}$ of Definition \ref{Fequivalence}(1) is compatible. The next step is to consider the question under which conditions $\sim_{\E}$ will be allegorical. To this end, we first give a simplified description of the relation $\sim_{\E}$, as follows. 

\begin{lemma}\label{st-sy}
	For a stable system $\E$ in $\C$, one has $(f,g)\sim_\E (h,k)$ if, and only if, there are $x,y\in\E$ such that the following diagram commutes.
\end{lemma}

\begin{equation}\label{exist xy}
\xymatrix{&D\ar[ld]_{f}\ar[rd]^{g}&\\ A &\cdot\ar[u]_{x}\ar[d]^{y}& B \\ & E \ar[lu]^{h}\ar[ru]_{k}&}		
\end{equation}

\begin{proof}
The sufficiency of the condition is trivial. Its necessity follows from the fact that every sequence $(f,g)\leq_{\E}(p,q)\geq_{\E}(h,k)$ may be replaced by a sequence $(f,g)\geq_{\E}(p',q')\leq_{\E}(h,k)$. Indeed, the first sequence yields the commutative diagram
\begin{equation}
\xymatrix{&D\ar[ld]_{f}\ar[rd]^{g}\ar[d]^z &\\ A &F\ar[l]_{p}\ar[r]^{q}& B \\ & E \ar[lu]^{h}\ar[ru]_{k}\ar[u]_{z'}&}		
\end{equation}
with $z,z'\in\E$. Then, by setting $(x, y)={\rm pb}(z, z')$, we have $x,y\in\E$ and, 
with 
$p'=fx$ and $q'=gx$, obtain a sequence of the second type. 
\end{proof}	

Calling a split epimorphism (= retraction) $r$ in $\C$ {\em effective} if there exists a morphism $r'$ such that $ (r,r')$ is a kernel pair in $\C$, we are ready to state a criterion that answers the question we stated above.
\begin{theorem}\label{sim_F-allegory}
For a stable system $\E$ in $\C$, the equivalence relation $\sim_\E$  is allegorical if, and only if, for every effective retraction $r$, there is a morphism $z\in \E$ such that $rz\in\E$.
\end{theorem}
\begin{proof}
Let $\sim_\E$ be allegorical  and $r$ be an effective retraction; then there exist morphisms $r', f$ such that $(r, r')$  is a kernel pair of $f$. Since by hypothesis 
 $(1, f) \sim_{\E} (f, f)\circ(1, f)=(r, fr')$, by Lemma \ref{st-sy} there are morphisms $ z, z' \in \E$ such that  $z'=rz$ and $fz'=fr'z$. In particular, $rz\in \E$, as desired.	
	
Conversely, given any morphism $f:A\to B$ in $\C$ with kernel pair $(r,r')$, the split epimorphism $r$ is effective. Then, by hypothesis,  
there exists a morphism $z\in \E$ with $r z \in \E$. Looking at the trivially commuting diagram 
\begin{equation}
 \xymatrix{&&A\ar[lld]_{1}\ar[rrd]^{f}&&\\  A&&\cdot\ar[u]^{r z}\ar[d]_{z}&&B  \\ && K \ar[llu]^{r}\ar[rru]_{fr'}&&}		
\end{equation}
we conclude  $(1, f) \sim_{\E} (f, f)\circ(1, f)$. This means that $\sim_{\E}$ is allegorical. 
\end{proof}

\begin{corollary}\label{F contains Ret}
	For a stable system $\E$ in a category $\C$ containing all split epimorphisms, 
	the equivalence relation $\sim_\E$ is allegorical, making the category
	${\sf Span}_{\E}(\C)$ an allegory.
\end{corollary}
 \begin{proof}
 Under the condition $\mathrm{SplitEpi}(\C)\subseteq \E$,
	for every effective retraction $r$ one has $1\in\E$ and $r1=r\in\E$, so that $\sim_\E$ is allegorical by Theorem \ref{sim_F-allegory}, making the quotient category an allegory, by Theorem \ref{Spanallegory}.
\end{proof}


 \begin{remark}
 The converse statement of Corollary \ref{F contains Ret} manifestly fails. In fact, in any category $\C$ with pullbacks that contains some split epimorphism that is not an isomorphism, one can find a stable system $\E$ not containing all split epimorphisms, for which $\sim_{\E}$ is nevertheless allegorical: just consider the class $\E$ of all {\em mono}morphisms in $\C$\,! Indeed, $\E$ is trivially a stable system that, by hypothesis, cannot contain all split epimorphism; but since any section $s$
	of a given retraction $r$ trivially belongs to $\E$, as well as $rs=1$, the relation $\sim_\E$ is allegorical.
\end{remark}
However, the condition of Corollary \ref{F contains Ret} does become necessary for $\sim_\E$ to be allegorical whenever $\E$ belongs to a stable factorization system $(\E,\M)$ in $\C$, as we show in the next theorem. For that it is useful to recall a well-known equivalent formulation of the condition that all retractions must lie in $\E$; we provide a proof of the equivalence since resorting to the standard source for it (the dual of Proposition 14.11  in \cite{ada-her-str}) would ask us to unnecessarily assume the existence of binary coproducts in $\C$.
\begin{lemma}\label{splitepilemma} 
For any factorization system $(\E,\M)$ in a category $\D$ with kernel pairs, one has the equivalence	
$$\mathrm{SplitEpi}(\D)\subseteq \E\iff \M\subseteq \mathrm{Mono}(\D).    $$
\end{lemma} 
\begin{proof}
$\Longrightarrow$: For the kernel pair $(m_1,m_2)$ of a morphism $m\in\M$, the projection $m_1$ is in $\M$ as a pullback of $m$, but as a retraction also in $\E$, by hypothesis. Consequently, $m_1$ is an isomorphism, which makes $m$ monic.

$\Longleftarrow$: For any retraction $r$, considering its factorization 
$r=me$ with $m\in\M$ and $e\in\E$, we see that also $m$ is a retraction, in addition to being monic by hypothesis. Hence, the morphism $m$, now shown to be an isomorphism, makes $r=me$ lie in $\E$.
\end{proof}

 \begin{theorem}\label{nec suf con rel be all}
For every stable factorization system $(\E,\M)$ of $\C$, the relation   
	$\sim_{\E}$ is allegorical if, and only if, every morphism in $\M$ is a monomorphism in $\C$; equivalently, if every retraction in $\C$ lies in $\E$.
\end{theorem}

\begin{proof}	
By Corollary \ref{F contains Ret} and Lemma \ref{splitepilemma}, only the necessity of the condition $\M\subseteq \mathrm{Mono}(\C)$ still needs to be proved. 
But when $\sim_{\E}$  is allegorical, for every $m\in\M$ one has  $(1,m)\sim_\E(m,m)\circ(1,m)$, 
so that with Lemma \ref{st-sy} one obtains morphisms $x,y$ in $\E$
making the diagram
\begin{equation}
\xymatrix{&&\cdot\ar[lld]_{1}\ar[rrd]^{m}&&\\  \cdot &&\cdot\ar[u]^{x}\ar[d]_{y}&&\cdot  \\ &&\cdot  \ar[llu]^{m_1}\ar[rru]_{mm_2}&& }
\end{equation}
commute, where  $(m_1,m_2)$ is a kernel pair for $m$. Since both, $m_1 y=x$ and $y$ 
lie in $\E$,  also $m_1$ does. But 
$m_1$ already belongs to $\M$ as a pullback of $m$ and is therefore an isomorphism.  
Consequently, $m$ is a monomorphism.\\
\end{proof}

So, with Theorems \ref{RelasquotientofSpan} and \ref{nec suf con rel be all} we obtain:

\begin{corollary}\label{alleg iff cotains ret}
For a category  $\C$  with a stable factorization system $(\E,\M)$, the category  	${\sf Span}_{\E}(\C)\cong\sf{Rel}_\M(\C)$ is an allegory if, and only if, $ \mathrm{SplitEpi}(\C)\subseteq\E  $; equivalently, if $\M \subseteq \mathrm{Mono}(\C)$.
	\end{corollary}

\begin{definition}\label{def E_o}
For every stable system $\E$ in $\C$, we denote by $\E_{\circ}$ the least stable system that contains both $\E$ and $\mathrm{SplitEpi}(\C)$. Since these classes are both stable under pullback, $\E_{\circ}$ is just the closure of the class $\E\cup\mathrm{SplitEpi}(\C)$ under composition in $\C$:
$$\E_{\circ}=(\E\cup\mathrm{SplitEpi}(\C))^{\mathrm c}.$$
\end{definition}

With Corollary \ref{F contains Ret}, Lemma \ref{splitepilemma}  and Theorem  \ref{RelasquotientofSpan} we conclude:

\begin{corollary} 
For every stable system $\E$ in $\C$,
the category $\mathsf{Span}_{\E_{\circ}}(\C)$ is an allegory, and when $\E$ belongs to a stable factorization system $(\E,\M)$ of $\C$, it may be presented as a quotient category of $\mathsf{Rel}_\M(\C)$, via
$$\sf{Rel}_\M(\C)\longrightarrow \sf{Span}_{\E_{\circ}}(\C),\;
(f,g)\longmapsto[f,g]_{\sim_{\E_{\circ}}}.$$
 When $\M$ is a class of monomorphisms, then $\E_{\circ}=\E$, and this functor is an isomorphism.
\end{corollary}

\section{Unitary tabular allegories arising from stable factorization systems }
In this section, for a compatible equivalence relation $\sim$ on $\C$, we explore the question when the quotient category ${\sf Span}_\sim (\C)$ is an allegory that is even unitary and tabular in the sense of \cite{fs}. We give a complete answer in the case when $\sim$ is induced by a stable system $\E$ in $\C$. 
Let us first recall that, in an allegory $\A$,
\begin{itemize}
\item a  morphism $r$   is called a {\em map} if $1 \leq r^{\circ}\cdot r$ and $r\cdot r^{\circ} \leq 1$; the maps in $\A$ are the morphisms of the subcategory $$\mathsf{Map}(\A)$$ of $\A$, which has   the same class of objects as $\A$;
\item a pair $(f,g)$ of maps {\em tabulates} a morphism $r$ if $r = g\cdot f^{\circ}$ and
$(f^{\circ}\cdot f) \wedge (g^{\circ}\cdot g) = 1$; in this case $r$ is called {\em tabular};  the allegory $\mathcal A$ is {\em tabular} if all of its morphisms are;
\item an object $E$ in $\mathcal A$ is a {\em partial unit} if $1_{E}$
is the maximum in $\mathcal A(E, E)$; if, in addition, for every object $A$ there is a morphism $r:A\to E$ 
such that $1_A\leq r^\circ \cdot r$, then $E$ is said to be a {\em unit}; $\A$ is {\em  unitary} if it has a unit.
\end{itemize}
In order to establish some necessary conditions, we first consider a compatible and allegorical equivalence relation on $\C$ and start off by affirming that a morphism of $\C$ becomes a map once transferred to the $\sim$\,-\,equivalence class of its graph (Definition \ref{defallegoricalrelation}), by the functor
$$\Gamma_{\sim}:\C\longrightarrow \sf{Span}_{\sim}(\C),\quad f\longmapsto [1,f]_{\sim}.$$

\begin{lemma}
	For every morphism $f:A\to B$ in $\C$, $[1_A,f]_{\sim}$ is a map in the allegory $\sf{Span}_{\sim}(\C)$.
\end{lemma}
\begin{proof}
Writing $[-,-]$ for $[-,-]_{\sim}$, we need to establish the inequalities $$[1_A,f]\circ [f,1_A] \leq [1_B,1_B]\quad \text{ and }\quad [1_A,1_A]\leq [f,1_A]\circ [1_A,f].$$
The LHS of the first inequality trivially computes to $[f,f]$, so it suffices to show the identity $[f,f]=[f,f]\wedge[1_B,1_B]$; and the RHS of the second inequality equals $[f_1,f_2]$ with  $(f_1,f_2)$ a kernel pair of $f$, so it suffices to show the identity $[1_A,1_A]=[f_1,f_2]\wedge[1_A,1_A]$. But the validity of the two needed identities may be seen easily, since the relevant meets may be displayed by the following two square diagrams. Indeed, as the morphisms $\langle 1_B,1_B\rangle$ and $\langle f_1,f_2\rangle$ are monic, they are both pullback squares, with diagonals $\langle f, f\rangle$ and $\langle 1_A,1_A\rangle$ respectively.
	\begin{center}
	$\xymatrix{A\ar[r]^{f}\ar[d]_{1_A} & B\ar[d]^{\langle 1_B,1_B\rangle}\\
			A\ar[r]_{\langle f,f\rangle \quad} & B\times B}$ \qquad	
		$\xymatrix{A\ar[r]^{1_A}\ar[d]_{\langle 1_A,1_A\rangle} & A\ar[d]^{\langle 1_A,1_A\rangle}\\
			A\times_BA\ar[r]_{\langle f_1,f_2\rangle } & A\times A}$
	\end{center}
\end{proof}
By the lemma, we may restrict the codomain of the functor $\Gamma_{\sim}$ and consider it as a functor $\C\longrightarrow\mathsf{Map}(\mathsf{Span}_{\sim}(\C))$. Its behaviour on pullbacks is described in the following proposition which we will use in Section 6.

 \begin{proposition}\label{pb iff tab}
	Let $\sim$ be a compatible and allegorical equivalence relation on $\mathsf{Span}(\C)$. 
	Then $\Gamma_{\sim}$ maps a pullback square
	\begin{center}
		$\xymatrix{\cdot\ar[d]_{k'}\ar[r]^{h'} &\cdot\ar[d]^{k}\\ \cdot\ar[r]_{h}&\cdot}$
	\end{center}
	in $\C$ to a pullback square in $\mathsf{Map}(\mathsf{Span}_{\sim}(\C))$	
	if the pair $(\Gamma_{\sim}\,k',\Gamma_{\sim}\,h')$ tabulates the morphism $(\Gamma_{\sim}\,k)^{\circ}\circ(\Gamma_{\sim}\,h)$  
	in $\mathsf{Span}_{\sim}(\C)$. This condition is not only sufficient, but also necessary, provided that the morphism $(\Gamma_{\sim}\,k)^{\circ}\circ(\Gamma_{\sim}\,h)$ has any tabulation in
	$\mathsf{Span}_{\sim}(\C)$.
		\end{proposition}
\begin{proof}

	Continuing to write $[-,-]$ for $[-,-]_\sim$, we first assume that the pair $([1,k'],[1,h'])$ tabulates $[1,k]^\circ\circ [1,h]$. In order to show the pullback property in $\mathsf{Map}(\mathsf{Span}_{\sim}(\C))$, we consider any maps $x,y$ in $\mathsf{Span}_{\sim}(\C)$ with $[1,h]\circ x=[1,k]\circ y$. Then, by \cite[2.146]{fs}, we get $y\circ x^{\circ}\leq[1,k]^{\circ}\circ [1, h]$, and by \cite[2.143]{fs}, this inequality suffices to guarantee the existence of a unique map $z$ with $[1,k']\circ z= x$ and $[1,k']\circ z=y$. This shows that $\Gamma_\sim$ preserves the given pullback in $\C$.	
	 	
	Conversely, suppose that the functor $\Gamma_{\sim}$ maps the given pullback in $\C$ to a pullback in $\mathsf{Map}(\mathsf{Span}_{\sim}(\C))$, and assume that $[1,k]^{\circ}\circ[1,h]$ has some tabulation $(x,y)$. Then, invoking \cite[2.143 and 2.146]{fs} again, we see that the square	
			\begin{center}
		$\xymatrix{\cdot\ar[d]_{x}\ar[r]^{y} &\cdot\ar[d]^{[1,k]}\\ \cdot\ar[r]_{[1,h]}&\cdot}$
	\end{center}
	is a pullback diagram in $\mathsf{Map}(\mathsf{Span}_{\sim}(\C))$ and, hence, isomorphic to the $\Gamma_\sim$-image of the given pullback in $\C$. Consequently, 
	$([1,k'],[1,h'])$ tabulates the morphism $[1,k]^{\circ}\circ[1,h]$. 	 
\end{proof}

The assertion of the previous proposition may be strengthened when the relation $\sim$ is finer than the relation $\sim_\E$ induced by a stable factorization system $(\E,\M)$, as we show with the following proposition.

\begin{proposition}\label{m tabulation}
Let $(\E,\M)$ be a stable factorization system in $\C$ and $\sim$ a compatible and allegorical equivalence relation on $\mathsf{Span}(\C)$ containing $\sim_{\E}$. Then the following two conditions on the allegory $\mathsf{Span}_\sim(\C)$ are equivalent:
\begin{itemize}
\item[{\rm (i)}] for all morphisms $m\in \M$, the pair $([1,m]_\sim,[1,m]_\sim)$ tabulates $[m,m]_\sim$;	
\item[{\rm (ii)}] for all spans $(f,g):A\to B$ in $\C$, the pair $([1,\pi_1m]_{\sim},[1,\pi_2m]_{\sim})$ tabulates $[f,g]_{\sim}$, with $\pi_1,\pi_2$ denoting product projections, and with $m\in \M$ arising from the $(\E,\M)$-factorization of $\langle f,g\rangle:\cdot\to A\times B$.
\end{itemize}
If either condition holds, then the functor
$\Gamma_{\sim}:\C\longrightarrow \sf{Map}(\sf{Span}_{\sim}(\C))$ preserves pullbacks.

\end{proposition}
\begin{proof}
(i)$\Longrightarrow$(ii): Let $\langle f,g\rangle =me$ be a factorization with $m\in \M$ and $e\in \E$. Then, since $\sim_\E\,\subseteq\,\sim$, one trivially has $[f, g]_\sim =[\pi_1m,\pi_2m]_{\sim}= [1,\pi_2m]_{\sim}\circ [1,\pi_1m]^{\circ}_{\sim}$. It therefore suffices to show that the meet
$[1,\pi_1 m]_{\sim}^{\circ}\circ [1,\pi_1 m]_{\sim}\wedge [1,\pi_2 m]_{\sim}^{\circ}\circ [1,\pi_2 m]_{\sim}$ is an identity morphism. But the LHS of this meet is represented by the kernel pair of $\pi_1m$, and its RHS by the kernel pair of $\pi_2m$ which, as a routine calculation in $\C$ shows,  makes the meet to be represented by the kernel pair $(m_1,m_2)$ of $m$. Now, with hypothesis (i), we obtain the desired identity
$$[m_1,m_2]_\sim=[1,m]_{\sim}^{\circ}\circ [1,m]_{\sim}=[1,m]_{\sim}^{\circ}\circ [1,m]_{\sim}\wedge [1,m]_{\sim}^{\circ}\circ [1,m]_{\sim}=1.$$ 

(ii)$\Longrightarrow$(i):
Consider $m\in \M$ and form the factorizations $\langle m,m\rangle= nc$ and $\pi_1n=n'c'$ with $c,c'\in\E$ and $n,n'\in\M$. Then $n'(c'c)=\pi_1\langle m,m \rangle=m$ is an $(\E,\M)$-factorization of $m$, which forces $c'c$ to be an isomorphism and, hence, makes $c$ a (split) monomorphism lying in $\E$. Consequently, in $\mathsf{Span}(\C)$ we have 
$$(c,1)\circ(1,c)=(1,1)\quad\text{and}\quad(1,c)\circ(c,1)=(c,c)\leq_\E(1,1),$$
so that $[1,c]_\sim$ becomes an isomorphism in $\mathsf{Span}_\sim(\C)$. Since trivially
$$(1,\pi_1n)\circ(1,c)=(1,m)=(1,\pi_2n)\circ(1,c),$$
we have the commutative diagram
\begin{equation}
\xymatrix{&&\cdot\ar[lld]_{[1,m]_\sim}\ar[dd]^{[1,c]_\sim}\ar[rrd]^{[1,m]_\sim}&&\\  \cdot &&&& \cdot\\ &&\cdot  \ar[llu]^{[1,\pi_1n]_{\sim}}\ar[rru]_{[1,\pi_2n]_{\sim}}&&} 		
 \end{equation}
in $\mathsf{Span}_\sim(\C)$, of which the lower span tabulates $[m,m]_\sim$, by hypothesis (ii). Therefore, with $[1,c]_\sim$ being an isomorphism, its upper span tabulates $[m,m]_\sim$ as well.

The preservation of pullbacks by $\Gamma_\sim$ remains to be shown. For that, given a pullback in $\C$ as in Proposition \ref{pb iff tab}, it suffices to prove that the pair $([1,k']_{\sim },[1,h']_{\sim})$ tabulates the morphism $[1,k]_{\sim} ^{\circ}\circ [1,h]_{\sim}$ which, trivially, equals $[k',h']_\sim$. By hypothesis (ii), that morphism is tabulated by the pair $([1,\pi_1m']_\sim,[1,\pi_2m']_\sim)$, where $m'\in\M$ belongs to the factorization $\langle k',h'\rangle = m'e'$ with $e'\in\E$. As $\langle k',h'\rangle$ is monic, so is $e'$, so that one concludes that $[1,e']_\sim$ is an isomorphism---just like for $[1,c]_\sim$ above. And similarly to the above argumentation one can finish the proof by just pointing to the trivially commuting diagram
\begin{equation}
\xymatrix{&&\cdot\ar[lld]_{[1,k']_\sim}\ar[dd]^{[1,e']_\sim}\ar[rrd]^{[1,h']_\sim}&&\\  \cdot &&&& \cdot\\ &&\cdot  \ar[llu]^{[1,\pi_1m']_{\sim}}\ar[rru]_{[1,\pi_2m']_{\sim}}&&} 		
 \end{equation}

\end{proof}


Given a stable factorization system $(\E,\M)$ in $\C$, the question remains how to refine the relation $\sim_\E$ such that the finer relation satisfies the equivalent conditions of Proposition \ref{m tabulation}. The following notion of conjugation of morphisms in $\M$ as provided in \cite{HST} turns out to be crucial to answer this question.
 \begin{definition}
 In the following commutative diagram in $\C$, let the back and front squares be pullbacks and $m\in\M$. We then call the (unique) morphism $m^*$ the {\em conjugate of $m$ along $f$ and $s$} and denote by $\M^*$ the closure under pullbacks of the class all morphisms that are presentable as conjugates of some $m\in \M$ along some morphisms $f$ and $s$. 
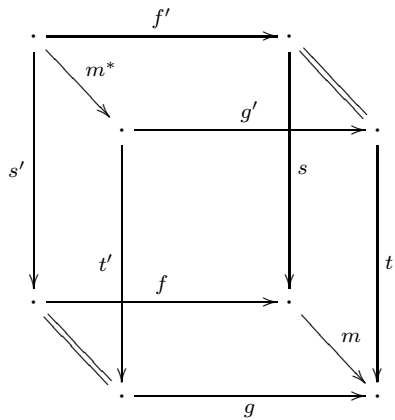
 \begin{center}
 	$\xymatrix{\cdot\ar[rrr]^{f'}\ar[ddd]_{s'}\ar[rd]^{m^{*}} &&&\cdot \ar@{=}[rd]^{}\ar[ddd]^{s} &\\
 		&\cdot \ar[rrr]^{g'}\ar[ddd]_{t'} &&&\cdot\ar[ddd]^{t}\\
 		&&&&\\
 		\cdot\ar[rrr]^{f}\ar@{=}[rd]_{} &&&\cdot\ar[rd]^{m} &\\
 		&\cdot\ar[rrr]_{g} &&&\cdot}$
 \end{center}
 \end{definition} 
Let us note immediately that the formation of the class $\M^*$ may be important only when $\M$ is {\em not} a class of monomorphisms in $\C$. Indeed, when $m$ is a monomorphism, the right panel of the cube is a pullback which, since the back and front panels are pullbacks, makes also the left panel a pullback. But then $m^*$ must be an isomorphism, as a pullback of an identity morphism. Consequently, $\M^*$ is just the class $\mathsf{Iso}(\C)$ when $\M\subseteq\mathsf{Mono}(\C)$.

Determining membership in $\M^*$ in the general case is a lot harder. But at least one can easily see that split monomorphisms with a retraction lying in $\M$ belong to $\M^*$, as follows:

\begin{lemma}\label{sec in M*}
	Suppose $rs=1$ with $r\in\M$. Then $s\in \mathcal M^*$.
\end{lemma}
\begin{proof}
	The diagram 
	\begin{center}
		$
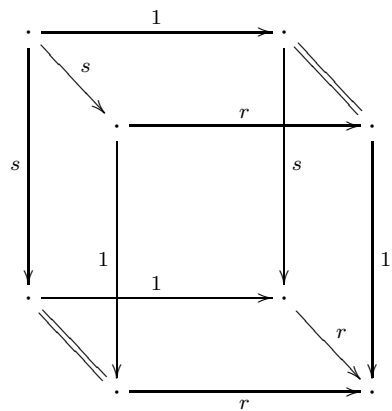
\xymatrix{\cdot\ar[rrr]^{1}\ar[ddd]_{s}\ar[rd]^{s} &&&\cdot \ar@{=}[rd]^{}\ar[ddd]^{s} &\\
			& \cdot\ar[rrr]^{r}\ar[ddd]_{1} &&&\cdot\ar[ddd]^{1}\\
			&&&&\\
			\cdot\ar[rrr]^{1}\ar@{=}[rd]_{} &&&\cdot\ar[rd]^{r} &\\
			&\cdot\ar[rrr]_{r} &&&\cdot}$
	\end{center}
	shows that $s$ is the conjugate of $r$ along $1$ and $s$. Thus $s\in\mathcal M^*$.
\end{proof}

Trading $\mathrm{SplitEpi}(\C)$ for $\M^*$ in Definition \ref{def E_o}, we define:

\begin{definition}
	For a stable factorization system $(\E,\M)$ in $\C$, we denote by $\E_{\bullet}$ the least stable system containing both $\E$ and $\M^*$:
	$$ \E_{\bullet}=(\E \cup \M^*)^{c}.$$
\end{definition}
We note that, when $\M$ is a class of monomorphisms in $\C$, so that $\M^*=\mathrm{Iso}(\C)$, one has $\E_{\bullet}=\E=\E_{\circ}$, but generally the classes $\E_{\bullet}$ and $\E_{\circ}$ may be incomparable by inclusion. Nevertheless, their induced equivalence relations {\em are} comparable by inclusion:


\begin{proposition}\label{relations A B}
Let $(\E , \M)$ be a stable factorization system in $\C$. Then the equivalence relation	$ \sim_{\E_{\circ}}$ is included in $ \sim_{\E_{\bullet}}$, and $\mathsf{Span}_{\E_{\bullet}}(\C)$ is an allegory. 
\end{proposition}
\begin{proof}
It suffices to show that $(f,g) \leq_{\E_{\circ}} (h,k)$ implies $(f,g) \sim_{\E_{\bullet}} (h,k)$. The hypothesis $(f,g) \leq_{\E_{\circ}} (h,k)$ gives us a morphism  $r\in \E_{\circ}$ such that 
$\langle f,g\rangle=\langle h,k\rangle r$. As $r$ is a composite of morphisms in $\E\cup\mathrm{SplitEpi}(\C)$, one may reason inductively and restrict oneself to the only non-trivial case, that is that $r$ is a retraction.
There is then a morphism $s$ with $rs=1$, and factoring $r=me$, with $m\in\M$ and $e\in\E$, we have $mes=1$. Thus, by Lemma
\ref{sec in M*}, $es\in \M^*$. Since $\langle h,k\rangle=\langle f,g\rangle s=\langle hm,km\rangle es$, this implies $(h,k)\leq_{\E_{\bullet}}( hm,km)$.
Since $\langle f,g\rangle=\langle hm,km\rangle e$ trivially gives $(f,g) \leq_{\E_{\bullet}} (hm,km)$, 
we can conclude that$(f,g) \sim_{\E_{\bullet}} (h,k)$ holds.

The relation $\sim_{\E_{\circ}}$ is allegorical by Corollary \ref{F contains Ret}, and one sees at once that this property gets inherited by any finer compatible equivalence relation. Hence, by Proposition \ref{relations A B},  $\sim_{\E_{\bullet}}$ is also allegorical, so that by Theorem \ref{almostallegory},  ${\sf Span}_{\E_{\bullet}}(\C)$ is an allegory. 
\end{proof}

In what follows, we are now ready to show that the allegory $\sf{Span}_{\E_{\bullet}}(\C)$ is even unitary and tabular. We denote the $\sim_{\E_{\bullet}}$-equivalence class of a span $(f,g)$ by $[f,g]_{\E_{\bullet}}$.

\begin{theorem}\label{tabulation}
For a finitely complete category $\C$ with a stable factorization structure $(\E , \M)$, the category	$\sf{Span}_{\E_{\bullet}}(\C)$ is a unitary tabular allegory. The terminal object in $\C$ serves as a unit in $\sf{Span}_{\E_{\bullet}}(\C)$, and the tabulation of a morphism 
	$[f,g]_{\E_{\bullet}}: A\to B$ is given by 
	$$ [f,g]_{\E_{\bullet}}	= [1,\pi_2m]_{{\mathcal E}_{\bullet}}\circ [1,\pi_1m]^{\circ}_{\E_{\bullet}},$$ where $m\in\M$ belongs to the $(\E,\M)$-factorization of the morphism 
	$\langle f,g\rangle:\cdot\to A\times B$, with $\pi_1$, $\pi_2$ denoting the product projections. 
\end{theorem}
\begin{proof}
As mentioned at the beginning of Section 3, the span $(\pi_1,\pi_2)$ given by the projections of $A\times B$ is obviously neutral with respect to $\wedge$ in the hom-set $\mathsf{Span}(\C)(A,B)$, for any objects $A,B$; this holds in particular for $A=B=E$ a terminal object in $\C$. But in this 
case the projections are isomorphisms, and the maximality of $1_E=[1_E,1_E]_{\E_{\bullet}}
=[\pi_1,\pi_2]_{\E_{\bullet}}$ in $ \sf{Span}_{\E_{\bullet}}(\C)(E,E)$
follows. Furthermore, for any object $A$ one has the morphism $r=[1_A,!_A]_{{\E}_{\bullet}}:A\to E$ which trivially satisfies $r^{\circ}\cdot r\geq1_A$. Hence, $E$ is a unit in  
$\sf{Span}_{\E_{\bullet}}(\C)$.

In order to prove that morphisms in $\sf{Span}_{\E_{\bullet}}(\C)$ may be tabulated as described in the theorem, 
	according to Proposition \ref{m tabulation} we just need to show that, for every $m\in \M$, the pair $([1,m]_{{\E}_{\bullet}},[1,m]_{{\E}_{\bullet}})$ tabulates $[m, m]_{\E_{\bullet}}$. Indeed, one trivially has	
	$[m, m]_{\E_{\bullet}}=[1, m]_{\E_{\bullet}}\circ [1, m]^{\circ}_{\E_{\bullet}}$, and 
	$$[m,1]_{\E_{\bullet}}\circ [1,m]_{\E_{\bullet}}\wedge[m,1]_{\E_{\bullet}}\circ [1,m]_{\E_{\bullet}}=[m,1]_{\E_{\bullet}}\circ [1,m]_{\E_{\bullet}}=	[m_1,m_2]_{\E_{\bullet}}$$
	is simply given by the kernel pair $(m_1,m_2)$ of $m$. By Lemma \ref{sec in M*}, the common section of $m_1,m_2$ lies in $\M^*\subseteq\E_{\bullet}$, so that 
	  $[m_1,m_2]_{\E_{\bullet}}=[1,1]_{\E_{\bullet}}$ follows. Hence, the allegory $\sf{Span}_{\E_{\bullet}}(\C)$ is tabular.
\end{proof}
Recalling that $\M \subseteq \mathrm{Mono}(\C)$ implies $\M^{*}=\mathrm{Iso}(\C)$ and then $\E_{\bullet}=\E=\E_{\circ}$, we may now augment the assertion of Corollary \ref{alleg iff cotains ret}, as follows:     

\begin{corollary}\label{nec suf con rel be tball}        
For a finitely complete category  $\C$  with a stable factorization system $(\E,\M)$, the category  	${\sf Span}_\E(\C)\cong\sf{Rel}_\M(\C)$ is an allegory if, and only if, $ \mathrm{SplitEpi}(\C)\subseteq\E  $ or, equivalently, $\M \subseteq \mathrm{Mono}(\C)$, and it is then even unitary and tabular.
	\end{corollary}



\section{Assigning to an allegory its regular category of maps is a right adjoint} 
We recall from \cite[A1 and A3]{js} that, for a unitary and tabular allegory, the category $\mathsf{Map}(\A)$ is a finitely complete and regular category, that is: it comes equipped with a stable (regular epi, mono)-factorization system. The regular epimorphisms (also known as {\em covers}) in $\mathsf{Map}(\A)$ are described as those maps $f$ in $\A$ with $ff^\circ =1$, and the monomorphisms as those satisfying $f^\circ f=1$. The unit of $\A$ serves as a terminal object in $\mathsf{Map}(\A)$.

In what follows we set up two (huge) 2-categories whose objects are respectively given by the unitary tabular allegories and the finitely complete categories equipped with a stable factorization system. They allow us to treat the formation of the map category of an allegory as a 2-functor
$$\mathsf{Map}:\mathfrak{UTabAll}\longrightarrow\mathfrak{StabFact},$$
and then to show that the assignment $(\C,\E,\M)\longmapsto \mathsf{Span}_{\E_{\bullet}}(\C)$ of Theorem 
\ref{tabulation} may be seen as the object part of its left adjoint, which we will call $\mathsf{All}$.

\medskip
In greater detail, in the 2-category $\mathfrak{UTabAll}$,
\begin{itemize}
	\item objects are unitary tabular allegories;
	\item morphisms $R:\A\longrightarrow\B$ are {\em unitary representations} (\cite[2.154]{fs}), {\em i.e.}, functors preserving units and commuting with the involutions and the meet operations;
	\item 2-cells $\phi:R\Longrightarrow S:\A\longrightarrow\B$ are natural transformations whose components are maps and whose naturality squares involving maps are pullbacks in $\mathsf{Map}(\B)$.
\end{itemize}
In the 2-category $\mathfrak{StabFact}$,
\begin{itemize}
\item objects $(\C,\E,\M)$ are given by finitely complete categories $\C$ that come equipped with a pullback-stable factorization system $(\E,\M)$;
\item morphisms $F:(\C,\E,\M)\longrightarrow(\D,\F,\N)$ are finite-limit preserving functors $F:\C\to\D$ with $F(\E)\subseteq\F$ and $F(\M)\subseteq\N$;
\item 2-cells $\theta:F\Longrightarrow G$ are natural transformations whose naturality squares are pullbacks in $\D$.
\end{itemize}
Of course, in both 2-categories, the horizontal and vertical compositions proceed as in the (huge) 2-category $\mathfrak{Cat}$ of all categories.

Since the notions of map and tabulation in an allegory are definable equationally with the operations $\circ$ and $\wedge$, a unitary representation $R:\A\longrightarrow  \B$ preserves these notions. Consequently, $R$ restricts to a functor $\mathsf{Map}(\A)\longrightarrow \mathsf{Map}(\B)$ which preserves the terminal object and pullbacks (these being given respectively by the unit and by tabulation) and, hence, all finite limits. Likewise, the restriction of $R$ to the map categories preserves regular epimorphisms and monomorphisms, as these are again characterized equationally via $\circ$. Finally, our definitions of 2-cells in the two huge 2-categories under consideration guarantee that we indeed have the 2-functor
$$\mathsf{Map}:\mathfrak{UTabAll}\longrightarrow\mathfrak{StabFact},\;\A\longmapsto(\mathsf{Map}(\A),\mathrm{RegEpi}(\mathsf{Map}(\A)),\mathrm{Mono}(\mathsf{Map}(\A))),$$
which operates on morphisms and 2-cells by restriction to the categories of maps. Next we establish the aforementioned 2-functor $\mathsf{All}$, going in the opposite direction.

\begin{proposition}\label{All}
There is a 2-functor
$$\mathsf{All}:\mathfrak{StabFact}\longrightarrow\mathfrak{UTabAll},\;(\C,\E,\M)\longmapsto\mathsf{Span}_{\E_{\bullet}}(\C).$$	
\end{proposition}

\begin{proof}
By Theorem \ref{tabulation}, $\mathsf{Span}_{\E_{\bullet}}(\C)$ is a unitary tabular allegory. For a morphism $F:(\C,\E,\M)\longrightarrow(\D,\F,\N)$ in $\mathfrak{StabFact}$, we first note that the inclusion $F(\M)\subseteq\N$ and the preservation of pullbacks imply $F(\M^*)\subseteq\N^*$ which, in conjunction with $F(\E)\subseteq\F$, gives $F(\E_{\bullet})\subseteq\F_{\bullet}$. Invoking pullback preservation once more, we see that $F$ induces a well-defined functor
$$\mathsf{All}(F):=\overline{F}:	\mathsf{Span}_{\E_{\bullet}}(\C)\longrightarrow\mathsf{Span}_{\F_{\bullet}}(\D),\;([f,g]_{\E_{\bullet}}:A\to B)\longmapsto([Ff,Fg]_{\F_{\bullet}}:FA\to FB).$$
Further, as required, the functor $\overline{F}$ obviously commutes with the involution, and it preserves the $\wedge$-operation as well as units since $F$ preserves finite products. Hence, $\overline{F}$ is a morphism in $\mathfrak{UTabAll}$.

Let us now consider a 2-cell $\theta:F\Longrightarrow G:(\C,\E,\M)\longrightarrow(\D,\F,\N)$ and prove that
$$\mathsf{All}(\theta):=\overline{\theta}:\overline{F}\Longrightarrow\overline{G}\quad\text{with}\quad\overline{\theta}_A:=[1,\theta_A]_{\F_{\bullet}}:{F}A\longrightarrow{G}A$$
defines a 2-cell in $\mathfrak{UTabAll}$. Clearly, the components of $\overline{\theta}$ are maps  in $\mathsf{Span}_{\F_{\bullet}}(\D)$. Next we observe that, since the naturality squares of $\theta$ are pullbacks in $\D$, for all spans $\xymatrix{A& D\ar[l]_f\ar[r]^g & B
}$ in $\C$ we have that the left square of the diagram
\begin{center}
$\xymatrix{FA\ar[rr]^{(Ff,Fg)}\ar[d]_{(1,\theta_A)} && FB\ar[d]^{(1,\theta_B)}\\
GA\ar[rr]^{(Gf,Gg)}&&GB}$\hfil
$\xymatrix	{FA\ar[rr]^{\overline{F}([f,g]_{\E_{\bullet}})}\ar[d]_{\overline{\theta}_A} && FB\ar[d]^{\overline{\theta}_B}\\
GA\ar[rr]^{\overline{G}([f,g]_{\E_{\bullet}})}&&GB}$
\end{center}
commutes in $\mathsf{Span}(\D)$, whence the right square commutes in $\mathsf{Span}_{\F_{\bullet}}(\D)$. This shows the naturality of $\overline{\theta}$, and we are left with having to prove that the right diagram is a pullback in $\mathsf{Map}(\mathsf{Span}_{\F_{\bullet}}(\D))$ whenever $[f,g]_{\E_{\bullet}}$ is a map in $\mathsf{Span}_{\E_{\bullet}}(\C)$. To this end, using again the fact that the naturality squares of $\theta$ are pullbacks, we note that the left square above decomposes as shown on the left of the following diagram, whence the right square above decomposes as shown on the right below:
\begin{center}
$\xymatrix{FA\ar[rr]^{(Ff,1)}\ar[d]_{(1,\theta_A)} && FD\ar[rr]^{(1,Fg)}\ar[d]^{(1,\theta_D)} && FB\ar[d]^{(1,\theta_B)}\\
GA\ar[rr]^{(Gf,1)} && GD\ar[rr]^{(1,Gg)} && GB
}$\hfil
$\xymatrix{FA\ar[rr]^{[Ff,1]_{\F_\bullet}}\ar[d]_{[1,\theta_A]_{\F_{\bullet}}} && FD\ar[rr]^{[1,Fg]_{\F_{\bullet}}}\ar[d]^{[1,\theta_D]_{\F_{\bullet}}} && FB\ar[d]^{[1,\theta_D]_{\F_{\bullet}}}\\
GA\ar[rr]^{[Gf,1]_{\F_{\bullet}}} && GD\ar[rr]^{[1,Gg]_{\F_{\bullet}}}&& GB}$
\end{center}
The right part of the double square on the right is a pullback, by Proposition \ref{m tabulation}, since it is the image of a naturality square of $\theta$ under the functor $\Gamma_{\sim_{\F_{\bullet}}}:\D\longrightarrow\mathsf{Map}(\mathsf{Span}_{\F_{\bullet}}(\D))$. Hence , it suffices to show that also its left part is a pullback, and for that it is sufficient to show that $[Ff,1]_{\F_{\bullet}}$ and $[Gf,1]_{\F_{\bullet}}$ are isomorphisms in $\mathsf{Map}(\mathsf{Span}_{\F_{\bullet}}(\D))$. In fact, by functoriality of $\overline{F}$ and $\overline{G}$, it suffices to show that $[f,1]_{\E_{\bullet}}$ is an isomorphism in $\mathsf{Map}(\mathsf{Span}_{\E_{\bullet}}(\C))$. But this last fact one sees easily: as $[f,g]_{\E_{\bullet}}$ is assumed to be a map, the pairs $([1,1]_{\E_{\bullet}},[f,g]_{\E_{\bullet}})$ and $([1,f]_{\E_{\bullet}},[1,g]_{\E_{\bullet}})$ both tabulate $[f,g]_{\E_{\bullet}}$, so the two tabulations can differ only by an isomorphism, which makes $[1,f]_{\E_{\bullet}}$ an isomorphism since $[1,1]_{\E_{\bullet}}$ is an isomorphism, and its inverse must be $[1,f]_{\E_{\bullet}}^{\circ}=[f,1]_{\E_{\bullet}}$ (see \cite[2.35]{fs}).
\end{proof}

We are now ready to put the assertion of Theorem \ref{tabulation} into a "global perspective":
 \begin{theorem}\label{globaltheorem}
 	There is a $2$-adjunction
 	$$\xymatrix{\mathsf{All}\dashv \mathsf{Map}:\mathfrak{UTabAll}\ar[rr]&& \mathfrak{StabFact}}$$
 	whose counit is an isomorphism. Hence, up to 2-equivalence, $\mathfrak{UTabAll}$ may be considered as a 2-reflective subcategory of $\mathfrak{StabFact}$.
 \end{theorem}
 \begin{proof}
 For an object $(\C,\E,\M)$ in $\mathfrak{StabFact}$, we define the $(\C,\E,\M)$-component of the desired 2-natural transformation $\Gamma:\mathrm{Id}\Longrightarrow \mathsf{Map}\mathsf{All} $ to be the (identity-on-objects) functor
 $$\Gamma_{\sim_{\E_{\bullet}}}:\C\longrightarrow \mathsf{Map}(\mathsf{Span}_{\E_{\bullet}}(\C)),\quad f\longmapsto[1,f]_{\E_{\bullet}},$$
 which, by Proposition \ref{m tabulation}, preserves pullbacks, as well as the terminal object (see Theorem \ref{tabulation}). Item (i) of Proposition \ref{m tabulation} also confirms that, for $m\in \M$, one has $[1,m]^{\circ}_{\E_{\bullet}}\circ[1,m]_{\E_{\bullet}}=1$, which makes $\Gamma_{\sim_{\E_{\bullet}}}m$ a monomorphism. For $e\in\E$, since $(e,e)\leq_{\E}(1,1)$, one trivially has $[1,e]\circ[1,e]^{\circ}_{\E_{\bullet}}=1$ and, therefore, that $\Gamma_{\sim_{\E_{\bullet}}} e$ is a regular epimorphism in $\mathsf{Map}(\mathsf{Span}_{\E_{\bullet}}(\C))$. This establishes  $\Gamma_{\sim_{\E_{\bullet}}}$ as a morphism in $\mathfrak{StabFact}$. 
  
In order to show the 2-naturality of $\Gamma$, one must confirm that, for all morphisms $F,\,
G:(\C,\E,\M)\longrightarrow(\D,\F,\N)$ and every 2-cell	$\theta:F\Longrightarrow G$, the diagram 
 	\begin{center}
 	$\xymatrix{(\C,\E,\M)\ar@{->}[rrr]^{\Gamma_{\sim_{\E_{\bullet}}}}\ar@<-1.5ex>[dd]_{F}^{\hspace{1mm} \Longrightarrow^{\hspace{-3mm}{^ \theta}}}\ar@<3ex>[dd]^{G} &&&
 	\mathsf{Map}(\sf{Span}_{\E_{\bullet}}(\C))\ar@<-4ex>[dd]^{\hspace{4mm} \Longrightarrow^{\hspace{-6mm} ^{\ \ \overline{\theta}}}}_{\overline{F}}\ar@<4ex>[dd]^{\overline{G}}\\
 		&&&\\
 		(\D,\F,\N)\ar@{->}[rrr]_{\Gamma_{\sim_{\F_{\bullet}}}}&&& \mathsf{Map} (\sf{Span}_{\E'_{\bullet}}(\D))}$
 \end{center}
 commutes in an obvious sense (where we have omitted listing the regular factorization systems of the RHS categories). But this confirmation involves only short calculations referring exclusively to the applicable definitions, which we can omit here. 
  	
 	 
 	 
 	 
 We now proceed to establish the prospective counit $\Delta:\mathsf{AllMap}\Longrightarrow\mathrm {Id}$ of the desired 2-adjunction. For that, we use the {\em Freyd-Scedrov Representation Theorem} for (unitary) tabular allegories \cite[2.148]{fs}: for every (unitary) tabular allegory, one has the (natural) tabulation isomorphism  
 $$\mathsf{Rel}(\mathsf{Map}(\A))\longrightarrow\A,\quad(h,k)\longmapsto k\cdot h^{\circ},$$
 which we may precompose with the isomorphism (in $\mathfrak{Cat}$; see Theorem \ref{RelasquotientofSpan})
 $$\mathsf{Span}_\R(\mathsf{Map}(\A))\longrightarrow\mathsf{Rel}(\mathsf{Map}(\A)),\quad[f,g]_\E\longmapsto(\pi_1m,\pi_2m),$$ 
 where $\R=\R_{\bullet}$ is the class $\mathrm{RegEpi}(\mathsf{Map}(\A))$ and $m$ is the regular image of $\langle f,g\rangle:\cdot\to  A\times B$.
 The composite of these two isomorphisms defines the $\A$-component of $\Delta$. It clearly commutes with the involutions and meet-operations of the participating allegories and is therefore an isomorphism in $\mathfrak{UTabAll}$. We can omit the proof that $\Delta$ is indeed a 2-natural transformation, which proceeds as routinely as the corresponding proof for $\Gamma$.

It remains for us to establish the triangular identities
 	   $$\mathsf{Map}\Delta\circ \Gamma\mathsf{Map}=1_{\mathsf{Map}} \quad\text{and}\quad \Delta \mathsf{All}\circ \mathsf{All}\,\Gamma=1_{\mathsf{All}}.$$
 	   For the first identity, we need to show that, for a unitary tabular allegory $\A$ and $\R=\mathrm{RegEpi}(\mathsf{Map}(\A)) $, the composite functor
 	   \begin{center}
 	   $\xymatrix{\mathsf{Map}(\A)\ar[rr]^{\Gamma_{\sim_\R}\qquad} && \mathsf{Map}(\mathsf{Span}_{\R }(\mathsf{Map}(\A)))\ar[rr]^{\qquad\mathsf{Map}(\Delta_\A)} && \mathsf{Map}(\A)
 	   }$
 	   \end{center}
 	   is the identity functor. Indeed, the first functor sends a map $f$ in $\A$ to  $[1,f]_\R $ which, as a map, tabulates itself trivially, and (the map-restriction of) $\Delta_\A$, as a morphism of allegories, preserves the trivial tabulation, so that one gets $f\cdot 1^{\circ}=f$.
 	   
 	   
 	   For the second identity, we consider $(\C, \E, \M)$ in $\mathfrak{StabFact}$ and show that the composite functor
 	    \begin{center}
 	    $\xymatrix{\mathsf{Span}_{\E_{\bullet}}(\C)\ar[rr]^{\mathsf{All}(\Gamma_{\E_{\bullet}})\qquad} && \mathsf{Span}_\R(\mathsf{Map}(\mathsf{Span}_{\E_{\bullet}}(\C)))\ar[rr]^{\qquad\Delta_{\mathsf{Span}_{\E_{\bullet}}(\C)}} && \mathsf{Span}_{\E_{\bullet}}(\C)  	    }$
 	    \end{center}
 	    maps every morphism 
 	   $[f, g]_{\E_{\bullet}}$ in $\mathsf{Span}_{\E_{\bullet}}(\C)$ identically. Indeed, the first functor maps $[f,g]_{\E_{\bullet}}$ to the morphism $[\Gamma_{\E_{\bullet}}(f),\Gamma_{\E_{\bullet}}(g)]_\R$, which then the second functor tabulates as
 	   $$\Gamma_{\E_{\bullet}}(g)\circ\Gamma _{\E_{\bullet}}(f)^{\circ}=[1,g]_{\E_{\bullet}}\circ[f,1]_{\E_{\bullet}}=[f,g]_{\E_{\bullet}},$$  
 	   as desired.  	   
 \end{proof}
 
 As an important consequence of the Theorem, we obtain that the image under the right-adjoint 2-functor $\mathsf{Map}$ is 2-equivalent to the 2-category $\mathfrak{UTabAll}$ and 2-reflective in $\mathfrak{StabFact}$. Furthermore, that image is precisely the (huge) 2-category
 $$\mathfrak{RegCat}$$
 of regular finitely complete categories, defined as a full sub-2-category of $\mathfrak{StabFact}$; its equivalence with $\mathfrak{UTabAll}$ was formulated explicitly (in the small one-dimensional version) in \cite[2.154]{fs} and shown comprehensively in \cite[Theorem 3.2.10]{js}.
\begin{corollary}
	The 2-category $\mathfrak{UTabAll}$ is 2-equivalent to $\mathfrak{RegCat}$, which is a 2-reflective full sub-2-category of $\mathfrak{StabFact}$.
\end{corollary}

Without reference to allegories, we can formulate explicitly the one-dimensional part of the universal property describing the reflectivity statement, as follows:

\begin{corollary}
	For every finitely complete category $\C$ with a (not necessarily proper) stable $(\E,\M)$-factorization system, there is a finite-limit-preserving functor $\Gamma:\C\to\D$ to a regular category $\D$ which maps morphisms in $\E$ to regular epimorphisms and morphisms in $\M$ to monomorphisms in $\D$, such that any other functor $\C\to\D'$ of this type factors through $\Gamma$, by a uniquely determined functor $\D\to\D'$ that preserves finite limits and regular epimorphisms. Moreover, the category $\D$ may be constructed to have the same objects as $\C$, mapped identically by $\Gamma$.
\end{corollary}

We end this article by presenting an easy example of a category $\C$ with a stable factorization system $(\E,\M)$ for which $\E\nsubseteq \mathrm{Epi}(\C)$ and $\M\nsubseteq \mathrm{Mono}(\C)$, such that the allegory $\A= \mathsf{All}(\C,\E,\M)$ admits a non-trivial representation $\A\longrightarrow\mathsf{Rel}$ to the allegory $\mathbf{Rel}=\mathsf{Rel}(\mathbf{Set})$ of sets and relations. 

\begin{example}
In the (ordinary) category $\C=\mathbf{Cat}$ of small categories we have the stable factorization system $(\E,\M) 
$ where  $\E$ stands for ``surjective on objects'' and $\M$ for ``fully faithful and injective on objects''.	Being right adjoint, the set-of-objects functor $\mathrm{ob}: \mathbf{Cat}\to\mathbf{Set}$ to the regular category $\mathbf{Set}$ preserves all limits, and it maps the classes $\E$ and $\M$ appropriately to make it a morphism in $\mathfrak{StabFact}$. Since $\mathbf{Set}\cong \mathsf{Map}(\mathbf{Rel})$, by the adjunction of Theorem \ref{globaltheorem} the functor $\mathrm{ob}$ corresponds to a representation $\A\to\mathbf{Rel}$, where the objects of the allegory $\A=\mathsf{All}(\mathbf{Cat},\E,\M)$ are those of $\mathbf{Cat}$. This representation maps objects like the functor $\mathrm{ob}$ does, and a morphism $[f,g]_{\E_{\bullet}}:A\to B$ in $\A=\mathsf{Span}_{\E_{\bullet}}(\mathbf{Cat})$ gets mapped to the relation given by the image of the function $\langle \mathrm{ob}\,f,\mathrm{ob}\,g\rangle$ in $\mathrm{ob}A\times\mathrm{ob}B$.

Consequently, the morphisms of $\A$ offer a new type of categorification of the notion of relation between sets.
\end{example}

 \refs
\bibitem[Adamek, Herrlich, Strecker 2004]{ada-her-str} J. Ad\'amek, H. Herrlich, G.E. Strecker. {\em Abstract and Concrete Categories}, John Wiley \& Sons, New York 1990. Online edition: 2004.
\bibitem[B\'enabou 1967]{Benabou} J. B\'enabou. {\em Introduction to Bicategories},  Lecture Notes in Mathematics 40, Springer, Berlin, 1967.
\bibitem[Borceux 1994]{Borceux1} F. Borceux. {\em Handbook of Categorical Algebra 1. Basic Category Theory}, Cambridge University Pres, 1994.
\bibitem[Freyd, Scedrov 1990]{fs} P. J. Freyd and A. Scedrov. {\em Categories, Allegories}, North Holland, Amsterdam 1990.
\bibitem[Hosseini, Shir Ali Nasab, Tholen 2020]{HST} S.N. Hosseini, A.R. Shir Ali Nasab, W. Tholen. Fraction, restriction and range categories from stable classes of morphisms, {\em Journal of Pure and Applied Algebra} 224(9):106361, 2020.
\bibitem[Jayewardene, Wyler 1996]{JW 1996} R. Jayewardene and O. Wyler. Categories of relations and functional relations, {\em Applied Categorical Structures} 8:279--305, 1996.
\bibitem[Johnstone 2002]{js} P. Johnstone. {\em Sketches of an Elephant: A Topos Theory Compendium}, vol. 1, Oxford University Press, Oxford 2002.
\bibitem[Klein 1970]{Klein} A. Klein. Relations in categories, {\em Illinois Journal of Mathematics}, 14:536--550, 1970.
\bibitem[Meisen 1974]{Meisen} J. Meisen. On bicategories of relations and pullback spans, {\em Communications in Algebra}, 1(5):377--401, 1974.
\bibitem[Milius 2000]{Milius} S. Milius. {\em Relations in Categories}, Master's thesis, York University, Toronto, Ontario, June 2000.
\bibitem[Pavlovi\'c 1995]{Pav 1}D. Pavlovi\'c. Maps I: Relative to a factorization system, {\em Journal of Pure and Applied Algebra} 99:9--34, 1995.
\end{document}